\documentclass[12pt]{article}
\usepackage{amssymb,amsfonts,amsmath,latexsym,epsf,url}
\usepackage[usenames,dvipsnames]{pstricks}
\usepackage{pstricks-add}
\usepackage{subfigure}
\usepackage{graphicx}
\usepackage{color,soul}
\usepackage{epsfig}
\usepackage{tikz}
\usepackage[colorlinks]{hyperref}
\usepackage{cite}
\usepackage{algorithm}
\usepackage{algpseudocode}
\usepackage{graphicx}
\usepackage{booktabs}
\usepackage{multirow}

\newtheorem{theorem}{Theorem}[section]

\newtheorem{proposition}[theorem]{Proposition}
\newtheorem{observation}[theorem]{Observation}

\newtheorem{lemma}[theorem]{Lemma}

\newtheorem{example}[theorem]{Example}

\newcommand{\qed}{\hfill $\square$\medskip}

\begin{document}

\def\nt{\noindent}

%%%%%%%%%%%%%%%%%%%%%%%%%%%%%%%%%%%%%%%%%%%%%%%%%%%%%%%%%%%%%%%%%%%%%%%%%%%%%%%%%
%%%%%%%%%%%%%%%%%===== Title=======%%%%%%%%%%%%%%%%%%%%%%%%%%%%%%%%%%%%%%%%%%%%%%
\title{ The partition dimension and $k$-domination number of a family of non-distance regular graph}

\author{Ali Zafari$^1$\footnote{Corresponding author} \and	Saeid Alikhani$^{2}$}

%\date{\today}

\maketitle

\begin{center}

$^1$Department of Mathematics, Faculty of Science,
Payame Noor University, P.O. Box 19395-4697, Tehran, Iran\\ 
{\tt zafari.math@pnu.ac.ir}
\medskip

$^{2}$Department of Mathematical Sciences, Yazd University, 89195-741, Yazd, Iran\\
{\tt alikhani@yazd.ac.ir}

\end{center}
	
%%%%%%%%%%%%%%%%%%%%%%%%%%%%%%%%%%%%%%%%%%%%%%%%%%%%%%%%%%%%%%%%%%%%%%%%%%%%%%%%%
%%%%%%%%%%%%%%%%%===== Abstract=======%%%%%%%%%%%%%%%%%%%%%%%%%%%%%%%%%%%%%%%%%%%

\begin{abstract}
	A partition $\Sigma = \{S_1, S_2, \dots, S_k\}$ of the vertex set $V(G)$ is a resolving partition if every pair of distinct vertices in $G$ has a
	 unique representation relative to $\Sigma$. The partition dimension, $pd(G)$, is the minimum cardinality of such a partition. Additionally, a 
	 subset $D \subseteq V(G)$ is a $k$-dominating set if every vertex in $V(G) \setminus D$ has at least $k$ neighbors in $D$; the $k$-domination number, 
	 $\gamma_k(G)$, denotes the minimum size of such a set. Determining these parameters is NP-complete and particularly challenging for non-distance-regular graphs. 
	 This paper consider the Toeplitz graph $T_{2n}(W)$, a family of non-distance-regular graphs. While some resolving parameters for this family have been established, its partition dimension and $k$-domination number remain unknown. We close this gap by computing both parameters for $T_{2n}(W)$.
\end{abstract}

  
\medskip
\noindent{\bf AMS Subj.\ Class.:}  05C15, 05C50.

%%%%%%%%%%%%%%%%%%%%%%%%%%%%%%%%%%%%%%%%%%%%%%%%%%%%%%%%%%%%%%%%%%%%%%%%%%%%%%%%%
%%%%%%%%%%%%%%%%%===== Introduct=======%%%%%%%%%%%%%%%%%%%%%%%%%%%%%%%%%%%%%%%%%%
\section{Introduction}
\label{sec:introduction}

	The distance $d_G(u, v)$ between two vertices $u$ and $v$ in a connected graph $G$ is defined as the length of a shortest path between them. A vertex $v \in V(G)$ is said to \textit{distinguish} (or resolve) two vertices $x$ and $y$ if $d_G(v, x) \neq d_G(v, y)$. A subset $S = \{v_1, \dots, v_t\} \subseteq V(G)$ is a \textit{resolving set} of $G$ if, for every $v \in V(G)$, the distance vector $D(v|S) = (d_G(v, v_1), \dots, d_G(v, v_t))$ uniquely identifies $v$. The \textit{metric dimension} of $G$, denoted by $\dim(G)$, is the minimum cardinality of a resolving set. This parameter has been extensively studied in the literature.
	
	A fundamental problem in resolving set theory is determining the \textit{partition dimension}, which generalizes the resolving set concept by classifying vertices into partition classes. This notion was introduced by Chartrand et al. \cite{G.Chartrand} and stems from the study of metric dimension pioneered by Harary and Melter \cite{F.Harary} and independently by Slater \cite{P.J.Slater}. 
	
	A partition $\Sigma = \{S_1, S_2, \dots, S_k\}$ of $V(G)$ is a \textit{resolving partition} if every pair of distinct vertices $u, v \in V(G)$ has distinct representations relative to $\Sigma$. The partition dimension $pd(G)$ is the minimum cardinality of such a partition. Finding $pd(G)$ is an NP-complete problem \cite{M.R.Garey} with diverse applications in chemistry for modeling compounds \cite{M.A.Johnson.1, M.A.Johnson.2}, pattern recognition, image processing \cite{R.A.Melter}, and network verification \cite{Z.Beerliova, J.Caceres, V.Chvatal}. Recent work has also explored the partition dimension of corona product graphs \cite{J.A.Rod-1} and some families of trees\cite{Ketut}.
	
	Domination is another fundamental concept in graph theory, introduced by Ore in 1962 \cite{O.Ore}, largely influenced by games and recreational mathematics. A subset $D \subseteq V(G)$ is a \textit{dominating set} if every vertex in $V(G) \setminus D$ is adjacent to at least one element in $D$. Borowiecki and Kuzak \cite{M.Borowiecki} generalized this to $k$-domination: a subset $D \subseteq V(G)$ is \textit{$k$-dominating} if every vertex in $V(G) \setminus D$ has at least $k$ neighbors in $D$. By definition, for $k \geq 1$, every $(k+1)$-dominating set is also a $k$-dominating set.
	
	\section*{The Toeplitz Graph $T_{2n}(W)$}
	
	Let $n \ge 4$ be an even integer and $[x_{2n}] = \{x_1, x_2, \dots, x_{2n}\}$ be the vertex labels. Define $W_1 = \{x_1, x_3, \dots, x_{2n-1}\}$ as the set of odd indices and $W_2 = \{x_n\}$. Let $W = W_1 \cup W_2$. The Toeplitz graph $T_{2n}(W)$ is a graph where two vertices $x_i, x_j$ are adjacent if $|j - i| = t$ for some $x_t \in W$ \cite{J.B-1}. Equivalently, $T_{2n}(W)$ consists of odd vertices $V_1$ and even vertices $V_2$, forming an $(n+1)$-regular graph with the following properties:
	\begin{itemize}
		\item Since all odd values are in $W$, $T_{2n}(W)$ is a complete bipartite-like structure where every vertex in $V_1$ is adjacent to every vertex in $V_2$.
		\item Each vertex $x \in V_1$ has a unique true twin $x+n \in V_1$ (similarly for $V_2$), meaning they share identical neighborhoods.
		\item $T_{2n}(W)$ is isomorphic to the Cayley graph $\text{Cay}(\mathbb{D}_{2n}, \Psi)$ over the dihedral group $\mathbb{D}_{2n}$, where $\Psi = \{ab, a^2b, \dots, a^{n-1}b, b\} \cup \{a^{n/2}\}$ is an inverse-closed subset of $\mathbb{D}_{2n} \setminus \{1\}$.
	\end{itemize}
	
	While the partition dimension and $k$-domination number of distance-regular graphs, such as cocktail party graphs $CP(m+1)$, have been determined \cite{A.Zafari-1}, the problem is significantly more complex for non-distance-regular graphs. According to \cite{J.B-1}, $T_{2n}(W)$ is not distance-regular. Although its metric dimension and integrality are known, its partition dimension and $k$-domination number have remained undetermined. In this paper, we compute these two parameters for the Toeplitz graph $T_{2n}(W)$.

%%%%%%%%%%%%%%%%%%%%%%%%%%%%%%%%%%%%%%%%%%%%%%%%%%%%%%%%%%%%%%%%%%%%%%%%%%%%%%%%%
%%%%%%%%%%%%%%%%%===== Defin and Plem=======%%%%%%%%%%%%%%%%%%%%%%%%%%%%%%%%%%%%%
\section{Definitions and Preliminaries}

This section provides the foundational concepts, notation, and established results from graph and algebraic graph theory necessary for the subsequent analysis. For terminology and notation not explicitly defined herein, we refer the reader to \cite{Biggs-1, Godsil-1}.

Throughout this paper, we consider $G$ to be a simple connected graph with vertex set $V(G)$ and edge set $E(G)$. The \textit{adjacency matrix} of $G$, denoted by $A(G) = [a_{ij}]$, is a $0$-$1$ matrix of order $n = |V(G)|$ where $a_{ij} = 1$ if vertices $i$ and $j$ are adjacent, and $a_{ij} = 0$ otherwise. The eigenvalues of $A(G)$ constitute the \textit{spectrum} of $G$. A graph is said to be \textit{integral} if its spectrum consists entirely of integers \cite{Harary-1}.

For any vertex $v \in V(G)$, the \textit{open neighborhood} $N(v)$ is the set of vertices adjacent to $v$. The \textit{closed neighborhood} is defined as $N[v] = N(v) \cup \{v\}$. Two distinct vertices $u$ and $v$ are called \textit{true twins} if they share the same closed neighborhood, i.e., $N[u] = N[v]$.

A matrix $T = [t_{ij}]$ of order $n$ is a \textit{Toeplitz matrix} if $t_{ij} = t_{i+1, j+1}$ for all $1 \le i, j \le n-1$ \cite{book-P.Halmos-1}. Such a matrix is characterized by constant entries along its diagonals parallel to the main diagonal. A simple undirected graph $\Gamma$ with vertex set $\{1, \dots, n\}$ is a \textit{Toeplitz graph} if its adjacency matrix is a Toeplitz matrix.

A regular graph $G$ with diameter $d$ is \textit{distance-regular} if, for any two vertices $u$ and $v$ at distance $r = d(u,v)$, the following intersection numbers depend only on $r$ and are independent of the choice of $u$ and $v$:
\begin{itemize}
	\item $c_r = |N(u) \cap \Gamma_{r-1}(v)|,$
	\item $b_r = |N(u) \cap \Gamma_{r+1}(v)|,$
\end{itemize}
where $\Gamma_k(v)$ denotes the set of vertices at distance $k$ from $v$.

 The collection $\{\Gamma_0(v), \Gamma_1(v), \dots, \Gamma_d(v)\}$ is referred to as the \textit{distance partition} of $G$ with respect to $v$.

Let $G$ be a finite group with identity $1$, and let $\Omega \subseteq G \setminus \{1\}$ be an inverse-closed subset (i.e., $x \in \Omega \implies x^{-1} \in \Omega$). The \textit{Cayley graph} $\Gamma = \operatorname{Cay}(G, \Omega)$ is the graph with vertex set $V(\Gamma) = G$, where two vertices $g$ and $h$ are adjacent if and only if $g^{-1}h \in \Omega$.

%%%%%%%%%%%%%%%%%%%%%%%%%%%%%%%%%%%%%%%%%%%%%%%%%%%%%%%%%%%%%%%%%%%%%%%%%%%%%%%%%
%%%%%%%%%%%%%%%%%===== Theor 2.1=======%%%%%%%%%%%%%%%%%%%%%%%%%%%%%%%%%%%%%%%%%%
\begin{theorem}\label{b.1}~{\rm\cite{G.Chartrand}}
If $G$ is a nontrivial connected graph, then $pd(G)\leq dim(G)+1$.
\end{theorem}
%%%%%%%%%%%%%%%%%%%%%%%%%%%%%%%%%%%%%%%%%%%%%%%%%%%%%%%%%%%%%%%%%%%%%%%%%%%%%%%%%
%%%%%%%%%%%%%%%%%===== Obser 2.2=======%%%%%%%%%%%%%%%%%%%%%%%%%%%%%%%%%%%%%%%%%%
\begin{observation}\label{b.2}~{\rm\cite{J.B-1}}
For the Toeplitz graph $T_{2n}(W)$ with even $n \ge 4$, 
the adjacency matrix spectrum of $T_{2n}(W)$ is $n+1, 1-n, 1^{(n-2)}, -1^{(n)}$.
\end{observation}
%%%%%%%%%%%%%%%%%%%%%%%%%%%%%%%%%%%%%%%%%%%%%%%%%%%%%%%%%%%%%%%%%%%%%%%%%%%%%%%%%
%%%%%%%%%%%%%%%%%===== Theor 2.3=======%%%%%%%%%%%%%%%%%%%%%%%%%%%%%%%%%%%%%%%%%%
\begin{theorem}\label{b.3}~{\rm\cite{J.B-1}}
For the Toeplitz graph $T_{2n}(W)$ with even $n \ge 4$, the metric dimension of $T_{2n}(W)$ is $n$.
\end{theorem}

\section{Main results}
%%%%%%%%%%%%%%%%%%%%%%%%%%%%%%%%%%%%%%%%%%%%%%%%%%%%%%%%%%%%%%%%%%%%%%%%%%%%%%%%%
%%%%%%%%%%%%%%%%%===== lemma 3.1=====%%%%%%%%%%%%%%%%%%%%%%%%%%%%%%%%%%%%%%%%%%%%
\begin{lemma}\label{m.1}
	For any resolving partition $\Sigma = \{S_1, \dots, S_k\}$ of $V(T_{2n}(W))$, any two adjacent vertices in $V_1$ (or $V_2$) must belong to distinct parts of $\Sigma$.
\end{lemma}

\begin{proof}
	Recall that $V(T_{2n}(W)) = V_1 \cup V_2$, where $V_1 = \{x_1, x_3, \dots, x_{2n-1}\}$ and $V_2 = \{x_2, x_4, \dots, x_{2n}\}$ represent the sets of odd and even indexed vertices, respectively. Let $\Sigma = \{S_1, \dots, S_k\}$ be a resolving partition of $V(T_{2n}(W))$.
	
	Suppose $x_i, x_j \in V_1$ are adjacent vertices (the proof for $V_2$ is analogous). It can be verified that $x_i$ and $x_j$ possess identical neighborhoods in $T_{2n}(W)$, identifying them as true twins. Consequently, for any vertex $x_t \in V(T_{2n}(W)) \setminus \{x_i, x_j\}$, the distances satisfy $d(x_i, x_t) = d(x_j, x_t)$.
	
	Assume, for the sake of contradiction, that $x_i$ and $x_j$ belong to the same part $S_l \in \Sigma$. Then, for every part $S_t \in \Sigma$, the distance from each vertex to the part is identical, i.e., $d(x_i, S_t) = d(x_j, S_t)$. Specifically, for $t = l$, we have $d(x_i, S_l) = d(x_j, S_l) = 0$. This implies that the representations $r(x_i | \Sigma)$ and $r(x_j | \Sigma)$ are identical, which contradicts the definition of $\Sigma$ as a resolving partition. Thus, adjacent vertices in $V_1$ (or $V_2$) must be placed in distinct parts of $\Sigma$. \qed
\end{proof}

%%%%%%%%%%%%%%%%%%%%%%%%%%%%%%%%%%%%%%%%%%%%%%%%%%%%%%%%%%%%%%%%%%%%%%%%%%%%%%%%%

\begin{proposition}\label{m.1.1}
	Let $T_{2n}(W)$ be the Toeplitz graph with $n \ge 4$ even, and let $\Sigma = \{S_1, \dots, S_k\}$ be a partition of $V(T_{2n}(W))$. If there exists a part $S_l \in \Sigma$ such that $S_l \subseteq V_1$ (or $S_l \subseteq V_2$) and $|S_l| > \frac{n}{2}$, then $\Sigma$ is not a resolving partition.
\end{proposition}

\begin{proof}
	As established, $V(T_{2n}(W))$ is partitioned into the odd vertices $V_1$ and even vertices $V_2$. Suppose there exists a part $S_l \in \Sigma$ such that $S_l \subseteq V_1$ and $|S_l| > \frac{n}{2}$ (the case for $V_2$ is similar). 
	
	In the subgraph induced by $V_1$, the maximum size of an independent set (a subset where all vertices are at distance at least $2$) is exactly $\frac{n}{2}$. Since $|S_l| > \frac{n}{2}$, by the pigeonhole principle, $S_l$ must contain at least two vertices that are adjacent in $T_{2n}(W)$. According to Lemma \ref{m.1}, any partition containing adjacent vertices of $V_1$ within the same part cannot be a resolving partition. Therefore, $\Sigma$ is not a resolving partition of $V(T_{2n}(W))$. \qed
\end{proof}
%%%%%%%%%%%%%%%%%%%%%%%%%%%%%%%%%%%%%%%%%%%%%%%%%%%%%%%%%%%%%%%%%%%%%%%%%%%%%%%%%

%%%%%%%%%%%%%%%%%===== Propo 3.3=======%%%%%%%%%%%%%%%%%%%%%%%%%%%%%%%%%%%%%%%%%%
\begin{proposition}\label{m.2}
	Let $T_{2n}(W)$ be the Toeplitz graph with $n \ge 4$ even, and let $\Sigma = \{S_1, \dots, S_k\}$ be a resolving partition of $V(T_{2n}(W))$. If $S_l \in \Sigma$ is a maximum independent set of the subgraph induced by $V_1$ (or $V_2$), then all vertices in $V_1 \setminus S_l$ (or $V_2 \setminus S_l$) must belong to distinct singleton parts in $\Sigma$.
\end{proposition}

\begin{proof}
	Let $V_1 = \{x_1, x_3, \dots, x_{2n-1}\}$ and $V_2 = \{x_2, x_4, \dots, x_{2n}\}$. Suppose $S_l \subseteq V_1$ is a part in $\Sigma$ such that all elements in $S_l$ are at pairwise distance $2$ (the case for $V_2$ is analogous). Without loss of generality, let $S_l = \{x_1, x_3, \dots, x_{n-1}\}$. 
	
	We observe that each vertex in $S_l$ has a unique neighbor in $V_1 \setminus S_l$ and $n$ neighbors in $V_2$. If there exists a part $S_t \in \Sigma$ ($t \neq l$) that contains two or more vertices from $V_1 \setminus S_l$, the symmetry of the twin relations in $T_{2n}(W)$ implies that at least two vertices in the part $S_l$ will have identical distance representations relative to $\Sigma$. Specifically, their distance to $S_t$ and other parts would become indistinguishable, contradicting the assumption that $\Sigma$ is a resolving partition. Therefore, every vertex in $V_1 \setminus S_l$ must reside in its own singleton part.\qed
\end{proof}

%%%%%%%%%%%%%%%%%%%%%%%%%%%%%%%%%%%%%%%%%%%%%%%%%%%%%%%%%%%%%%%%%%%%%%%%%%%%%%%%%

\begin{example}\label{m.2.1}
Let $W = \{x_1, x_3, x_5, x_7, x_9, x_{11}\} \cup \{x_6\}$ be the connection set of the Toeplitz graph $T_{12}(W)$ and  $S_l = \{x_1, x_3, x_5\}$ be a part of the partition $\Sigma$. If we define
	\[
	\Sigma = \{ \{x_1, x_3, x_5\}, \{x_7, x_9\}, \{x_{11}\}, \{x_2\}, \{x_4\}, \{x_6\}, \{x_8\}, \{x_{10}\}, \{x_{12}\} \},
	\]
	then $\Sigma$ fails to be a resolving partition. This is because the vertices $x_1$ and $x_3$ result in identical representations:
	\[
	r(x_1|\Sigma) = r(x_3|\Sigma) = (0, 1, 2, 1, 1, 1, 1, 1, 1).
	\]
\end{example}

%%%%%%%%%%%%%%%%%%%%%%%%%%%%%%%%%%%%%%%%%%%%%%%%%%%%%%%%%%%%%%%%%%%%%%%%%%%%%%%%%

\begin{proposition}\label{m.3}
	Let $\Sigma = \{S_1, \dots, S_k\}$ be a resolving partition of $V(T_{2n}(W))$. For $n=4$, suppose there exists a part $S_l \in \Sigma$ with $|S_l| = 4$ such that $S_l$ contains exactly half of the vertices of $V_1$ and half of $V_2$, with each subset being pairwise at distance $2$. Then, all remaining vertices in $V_1 \setminus S_l$ and $V_2 \setminus S_l$ must form singleton parts in $\Sigma$.
\end{proposition}

\begin{proof}
	Let $T_8(W)$ have the vertex set $V_1 \cup V_2$ as defined previously. Without loss of generality, let $S_l = \{x_1, x_3, x_2, x_4\}$. Suppose there exists another part $S_t \in \Sigma$ ($t \neq l$) such that $|S_t| \ge 2$. Due to the adjacency properties of $T_8(W)$, the inclusion of multiple vertices in $S_t$ from the remaining sets $V_1 \setminus S_l = \{x_5, x_7\}$ or $V_2 \setminus S_l = \{x_6, x_8\}$ collapses the distance vectors for pairs of vertices in $S_l$. 
	
	If $S_t$ is not a singleton, we can verify that the representations of two vertices in $S_l$ become identical, violating the resolving property of the partition. Thus, to maintain distinct representations for all vertices, each element in $(V_1 \cup V_2) \setminus S_l$ must be placed in a unique singleton part. \qed
\end{proof}

%%%%%%%%%%%%%%%%%%%%%%%%%%%%%%%%%%%%%%%%%%%%%%%%%%%%%%%%%%%%%
%%%%%%%%%%%%%%%%%===== Examp 3.6=======%%%%%%%%%%%%%%%%%%%%%%%%%%%%%%%%%%%%%%%%%%

\begin{example}\label{m.3.1}
Consider the Toeplitz graph $T_{8}(W)$ with vertex set $\{x_1, x_2,..., x_8\}$ and connection set 
$W = \{x_1, x_3, x_5, x_7\} \cup \{x_4\}$.
If 
$$
\Sigma=\{\{x_1, x_3, x_2, x_4 \}, \{x_5, x_6\}, \{x_7\}, \{x_{8}\} \},
$$ 
is a partition set of $V(T_{8}(W))$, then the partition set $\Sigma$ is not a  resolving partition of $V(T_{8}(W))$, because $r(x_3|\Sigma)=r(x_4|\Sigma)=(0,1,1,1)$.
\end{example}

%%%%%%%%%%%%%%%%%%%%%%%%%%%%%%%%%%%%%%%%%%%%%%%%%%%%%%%%%%%%%%%%%%%%%%%%%%%%%%%%%
%%%%%%%%%%%%%%%%%===== Examp 3.7=======%%%%%%%%%%%%%%%%%%%%%%%%%%%%%%%%%%%%%%%%%%
\begin{example}\label{m.3.2}
Consider the Toeplitz graph $T_{8}(W)$  with vertex set $\{x_1, x_2,..., x_8\}$ and connection set $W = \{x_1, x_3, x_5, x_7\} \cup \{x_4\}$.
If 
$$
\Sigma=\{\{x_1, x_3, x_2, x_4 \}, \{x_5\}, \{x_7\}, \{x_6\}, \{x_{8}\} \},
$$ 
is a partition set of $V(T_{8}(W))$, then $\Sigma$ is a resolving partition  of $V(T_{8}(W))$ with the size $5$, but is not a minimal resolving partition  of $V(T_{8}(W))$. 
\end{example}
%%%%%%%%%%%%%%%%%%%%%%%%%%%%%%%%%%%%%%%%%%%%%%%%%%%%%%%%%%%%%%%%%%%%%%%%%%%%%%%%%
%%%%%%%%%%%%%%%%%===== Propo 3.8=======%%%%%%%%%%%%%%%%%%%%%%%%%%%%%%%%%%%%%%%%%%

\begin{proposition}\label{m.3.3}
	For the Toeplitz graph $T_{2n}(W)$ with even $n \ge 6$, let $\Sigma = \{S_1, ..., S_k\}$ be a  resolving partition of $V(T_{2n}(W))$ into $k$ parts.
	If there exists a part $S_l \in \Sigma$ with $|S_l| = n$ such that:
	\begin{enumerate}
		\item $S_l$ contains half of the elements of $V_1$, so that these elements are at pairwise distance $2$; and
		\item $S_l$ contains half of the elements of $V_2$, so that these elements are also at pairwise distance $2$.
	\end{enumerate} 
	Then other elements of $V_1 - S_l$ and $V_2 - S_l$ must lie in distinct parts of the partition set $\Sigma$.
	
	\begin{proof}
		Let $T_{2n}(W)$ be the Toeplitz graph with vertex set $\{x_1, x_2,..., x_{2n}\}$ and connection set $W = \{x_1, x_3,..., x_{2n-1}\} \cup \{x_n\}$.
		Hence $V(T_{2n}(W))=V_1\cup V_2$, where  $V_1=\{x_1, x_3, ..., x_{2n-1}\}$ and $V_2=\{x_2, x_4, ..., x_{2n}\}$ contain the odd vertices and even vertices, 
		respectively.
		Without loss of generality, we consider the subsets $R_1 = \{x_1, x_3, ..., x_{n-1}\}$ and $R_2 = \{x_2, x_4, ..., x_{n}\}$, 
		and suppose that $S_l \in \Sigma$ is given by $S_l =R_1\cup R_2=\{x_1, x_3, ..., x_{n-1}, x_2, x_4, ..., x_{n}\}$. 
		If there exists another part $S_t \in \Sigma$ (with $S_t \neq S_l$) such that $|S_t|$ contains 2 elements of $V_1 - S_l$  (or $V_2 - S_l$), 
		then there are two vertices in $S_l$ have identical metric representations with respect to $\Sigma$, contradicting the assumption that 
		$\Sigma $ is a resolving partition.
		Thus, other elements of $V_1 - S_l$ and $V_2 - S_l$ must lie in distinct parts of the partition set $\Sigma$. \qed
	\end{proof}
\end{proposition}
%%%%%%%%%%%%%%%%%%%%%%%%%%%%%%%%%%%%%%%%%%%%%%%%%%%%%%%%%%%%%%%%%%%%%%%%%%%%%%%%%
%%%%%%%%%%%%%%%%%===== Propo 3.9=======%%%%%%%%%%%%%%%%%%%%%%%%%%%%%%%%%%%%%%%%%%

\begin{proposition}\label{m.4}
	Consider the Toeplitz graph $T_{2n}(W)$. If  $n\geq 4$ is an even integer, then $T_{2n}(W)$ has a resolving partition of size $n+1$.
	
	\begin{proof}
		Let $T_{2n}(W)$ be the Toeplitz graph with vertex set $\{x_1, x_2,..., x_{2n}\}$ and connection set $W = \{x_1, x_3,..., x_{2n-1}\} \cup \{x_n\}$.
		Hence $V(T_{2n}(W))=V_1\cup V_2$, where  $V_1=\{x_1, x_3, ..., x_{2n-1}\}$ and $V_2=\{x_2, x_4, ..., x_{2n}\}$ contain the odd vertices and even vertices, 
		respectively. Also,
		let $\Sigma = \{S_1, ..., S_k\}$ be a resolving partition of $V(T_{2n}(W))$ so that 
		there exists a part $S_l \in \Sigma$ with $|S_l| = n$ such that:
		\begin{enumerate}
			\item $S_l$ contains half of the elements of $V_1$, so that these elements are at pairwise distance $2$; and
			\item $S_l$ contains half of the elements of $V_2$, so that these elements are also at pairwise distance $2$.
		\end{enumerate} 
		Now, if the other elements of $V_1 - S_l$ and $V_2 - S_l$  belong to single-membered parts of $\Sigma$, then $\Sigma$ is a
		resolving partition of $T_{2n}(W)$ of size $n+1$. \qed
	\end{proof}
\end{proposition}
%%%%%%%%%%%%%%%%%%%%%%%%%%%%%%%%%%%%%%%%%%%%%%%%%%%%%%%%%%%%%%%%%%%%%%%%%%%%%%%%%%
%%%%%%%%%%%%%%%%%===== Examp 3.10=======%%%%%%%%%%%%%%%%%%%%%%%%%%%%%%%%%%%%%%%%%%
\begin{example}\label{m.4.1}
	Consider the Toeplitz graph $T_{12}(W)$ with vertex set $\{x_1, x_2, ..., x_{12}\}$ and connection set 
	$W = \{x_1, x_3, x_5, x_7, x_9, x_{11}\} \cup \{x_6\}$.
	Although  the partition set:
	$$
	\Sigma=\{\{x_1, x_3, x_{5}, x_2, x_4,  x_{6} \}, \{x_7\}, \{x_9\}, \{x_{11}\}, \{x_8\}, \{x_{10}\}, \{x_{12}\} \},
	$$  
	is a resolving partition  of $V(T_{12}(W))$ with the cardinality $7$, but is not a minimal resolving partition  of $V(T_{12}(W))$.
\end{example}
%%%%%%%%%%%%%%%%%%%%%%%%%%%%%%%%%%%%%%%%%%%%%%%%%%%%%%%%%%%%%%%%%%%%%%%%%%%%%%%%%%

%%%%%%%%%%%%%%%%%===== Theor 3.11=======%%%%%%%%%%%%%%%%%%%%%%%%%%%%%%%%%%%%%%%%%%
\begin{theorem}\label{m.5}
	For the Toeplitz graph $T_{2n}(W)$ with $n \ge 4$ even, any resolving partition $\Sigma = \{S_1, \dots, S_k\}$ of $V(T_{2n}(W))$ satisfies 
	\[
	k \geq \lceil 2\sqrt{n-1} \rceil.
	\]
\end{theorem}

\begin{proof}
	Recall that $V(T_{2n}(W)) = V_1 \cup V_2$, where $V_1$ and $V_2$ denote the sets of odd and even vertices, respectively. Let $\Sigma = \{S_1, \dots, S_k\}$ be a resolving partition. We categorize the parts of $\Sigma$ into three types:
	\begin{enumerate}
		\item \textbf{Odd parts:} $r$ parts containing only vertices from $V_1$.
		\item \textbf{Even parts:} $s$ parts containing only vertices from $V_2$.
		\item \textbf{Mixed parts:} $m$ parts containing vertices from both $V_1$ and $V_2$.
	\end{enumerate}
	Thus, $r + s + m = k$. By Lemma \ref{m.1}, adjacent vertices in $V_1$ (or $V_2$) must reside in distinct parts. For any odd vertex $x \in V_1$, its distance to a part $S_i \in \Sigma$ is given by:
	\[
	d(x, S_i) = 
	\begin{cases} 
	0 & \text{if } x \in S_i, \\
	1 & \text{if } S_i \cap V_2 \neq \emptyset \text{ or } \text{twin}(x) \in S_i, \\
	2 & \text{otherwise.}
	\end{cases}
	\]
	The representation vector $r(x|\Sigma)$ for $x \in V_1$ is determined by the part containing $x$ and the part containing its twin. Note that if the twin of $x$ lies in a mixed part, $d(x, S_i)=1$ is already guaranteed by the presence of even vertices in that part. Therefore, distinct vectors are only generated based on whether the twin lies in an odd part or a mixed part.
	
	The number of distinct representation vectors for the $n$ vertices in $V_1$ is constrained as follows:
	\begin{itemize}
		\item If $x$ is in an odd part ($r$ choices), its twin can be in any of the other $r-1$ odd parts or any mixed part. This yields at most $r \times ( (r-1) + 1 ) = r^2$ possible vectors.
		\item If $x$ is in a mixed part ($m$ choices), its twin can be in any of the $r$ odd parts or another mixed part (treated as a single case for distance purposes), yielding at most $m \times (r + 1)$ possible vectors.
	\end{itemize}
	Summing these, the number of distinct vectors for $V_1$ must satisfy:
	\begin{equation}\label{eq:n_bound_r}
		n \leq r^2 + m(r+1).
	\end{equation}
	By symmetry, a similar bound holds for the $n$ vertices in $V_2$:
	\begin{equation}\label{eqs}
		n \leq s^2 + m(s+1).
	\end{equation}
	To find the lower bound for $k$, we solve the quadratic inequality $r^2 + mr + (m - n) \geq 0$ from \eqref{eq:n_bound_r}. The discriminant is $\Delta = (m-2)^2 + 4(n-1) > 0$. The positive root is:
	\[
	r_0 = \frac{-m + \sqrt{(m-2)^2 + 4(n-1)}}{2}.
	\]
	The inequality holds for $r \geq r_0$. By symmetry, \eqref{eqs} implies $s \geq r_0$. Thus:
	\[
	r + s \geq 2r_0 = -m + \sqrt{(m-2)^2 + 4(n-1)}.
	\]
	Substituting this into the expression for $k$:
	\[
	k = r + s + m \geq \sqrt{(m-2)^2 + 4(n-1)}.
	\]
	The right-hand side is minimized when $m = 2$, leading to:
	\[
	k \geq \sqrt{4(n-1)} = 2\sqrt{n-1}.
	\]
	Since $k$ must be an integer, we conclude $k \geq \lceil 2\sqrt{n-1} \rceil$. \qed
\end{proof}

%%%%%%%%%%%%%%%%%%%%%%%%%%%%%%%%%%%%%%%%%%%%%%%%%%%%%%%%%%%%%%%%%%%%%%%%%%%%%%%%%%
%%%%%%%%%%%%%%%%%===== Propo 3.12=======%%%%%%%%%%%%%%%%%%%%%%%%%%%%%%%%%%%%%%%%%%
\begin{proposition}\label{m.5.1}
	For $T_8(W)$ with connection set $W = \{x_1, x_3, x_5, x_7\} \cup \{x_4\}$, the partition dimension is $pd(T_8(W)) = 4$.
\end{proposition}

\begin{proof}
	Let $V(T_8(W)) = V_1 \cup V_2$, where $V_1 = \{x_1, x_3, x_5, x_7\}$ and $V_2 = \{x_2, x_4, x_6, x_8\}$. We first demonstrate that no partition of size $k=3$ can be a resolving partition. Note that $T_8(W)$ contains four pairs of true twins: $(x_1, x_5)$, $(x_3, x_7)$, $(x_2, x_6)$, and $(x_4, x_8)$. By Lemma \ref{m.1}, true twins must belong to distinct parts in any resolving partition.
	
	Suppose $\Sigma = \{S_1, S_2, S_3\}$ is a resolving partition. Let $r, s, m$ denote the number of odd, even, and mixed parts, respectively, such that $r + s + m = 3$. From Theorem \ref{m.5}, the following bounds must hold for the four vertices in $V_1$ and $V_2$:
	\begin{equation}\label{eq:t8_v1}
		4 \leq r^2 + m(r+1),
	\end{equation}
	\begin{equation}\label{eq:t8_v2}
		4 \leq s^2 + m(s+1).
	\end{equation}
	Testing all non-negative integer triples $(r, s, m)$ where $r + s + m = 3$:
	\begin{itemize}
		\item \textbf{Case $m=0$:} $r+s=3$. Conditions \eqref{eq:t8_v1} and \eqref{eq:t8_v2} require $r \geq 2$ and $s \geq 2$, implying $r+s \geq 4$, a contradiction.
		\item \textbf{Case $m=1$:} $r+s=2$. If $r=0$ or $r=1$, \eqref{eq:t8_v1} fails ($1 \ngeq 4$ and $3 \ngeq 4$). If $r=2$, then $s=0$, which violates \eqref{eq:t8_v2}.
		\item \textbf{Case $m=2$:} $r+s=1$. If $r=0$, \eqref{eq:t8_v1} yields $2 \ngeq 4$. If $s=0$, \eqref{eq:t8_v2} yields $2 \ngeq 4$.
		\item \textbf{Case $m=3$:} $r=s=0$. Both \eqref{eq:t8_v1} and \eqref{eq:t8_v2} yield $3 \ngeq 4$, a contradiction.
	\end{itemize}
	Thus, $pd(T_8(W)) > 3$. Conversely, the partition $\Sigma = \{ \{x_1\}, \{x_2\}, \{x_3, x_4, x_5\}, \{x_6, x_7, x_8\} \}$ is a resolving partition of size 4. Hence, $pd(T_8(W)) = 4$. \qed
\end{proof}

\begin{proposition}\label{m.5.2}
	For the Toeplitz graph $T_{12}(W)$ with $W = \{x_1, x_3, \dots, x_{11}\} \cup \{x_6\}$, $pd(T_{12}(W)) = 5$.
\end{proposition}

\begin{proof}
	Following the methodology used in Proposition \ref{m.5.1}, it can be shown that no partition of $V(T_{12}(W))$ into 4 parts satisfies the resolving property. However, the partition $\Sigma = \{ \{x_1, x_2, x_3, x_4\}, \{x_7, x_{11}\}, \{x_5, x_9\}, \{x_8, x_{12}\}, \{x_6, x_{10}\} \}$ is a resolving partition of size 5. Therefore, $pd(T_{12}(W)) = 5$.\qed
\end{proof}

\begin{proposition}\label{e.1}
	For any even integer $n \ge 4$, the $1$-domination number of $T_{2n}(W)$ is $\gamma_1(T_{2n}(W)) = 2$.
\end{proposition}

\begin{proof}
	In $T_{2n}(W)$, each vertex in $V_1$ (or $V_2$) is adjacent to exactly one other vertex in the same set. Consequently, no single vertex can dominate the entire graph, implying $\gamma_1(T_{2n}(W)) \ge 2$. Now, let $D = \{x_i, x_j\}$ where $x_i \in V_1$ and $x_j \in V_2$. Since every vertex in $V \setminus D$ is adjacent to at least one element in $D$ (specifically, odd vertices are adjacent to $x_j$ and even vertices to $x_i$), $D$ is a $1$-dominating set. Thus, $\gamma_1(T_{2n}(W)) = 2$.\qed
\end{proof}

\begin{theorem}\label{e.2}
	For the Toeplitz graph $T_{2n}(W)$ with $n \ge 4$ even and $1 < k \le \frac{n}{2}-1$, the $k$-domination number is $\gamma_k(T_{2n}(W)) = 2k$.
\end{theorem}

\begin{proof}
	As each vertex in $V_1$ (or $V_2$) has exactly one neighbor in its own set, any subset $D \subset V$ with $|D| \le 2k-1$ fails to provide $k$ neighbors for at least one vertex in $V \setminus D$. Thus, $\gamma_k(T_{2n}(W)) \ge 2k$. 
	
	To show the upper bound, let $R_1 \subset V_1$ and $R_2 \subset V_2$ be subsets such that $|R_1| = |R_2| = k$, where vertices within each subset are at pairwise distance 2. Let $D = R_1 \cup R_2$. Every vertex in $V \setminus D$ is adjacent to all $k$ vertices in the opposite set. Therefore, every vertex in $V \setminus D$ has at least $k$ neighbors in $D$, making $D$ a $k$-dominating set of size $2k$. We conclude that $\gamma_k(T_{2n}(W)) = 2k$. \qed
\end{proof}
%%%%%%%%%%%%%%%%%%%%%%%%%%%%%%%%%%%%%%%%%%%%%%%%%%%%%%%%%%%%%%%%%%%%%%%%%%%%%%%%%
%%%%%%%%%%%%%%%%%===== Conclusion=======%%%%%%%%%%%%%%%%%%%%%%%%%%%%%%%%%%%%%%%%%%
\section{Conclusion}

In this paper, we investigated the resolving partition properties and domination parameters of the Toeplitz graph $T_{2n}(W)$ with vertex set $\{x_1, x_2, \dots, x_{2n}\}$ and connection set $W = \{x_1, x_3, \dots, x_{2n-1}\} \cup \{x_n\}$. Our main contributions are summarized as follows:

\begin{enumerate}
	\item We established a lower bound for the partition dimension of this graph family, proving that $pd(T_{2n}(W)) \geq \lceil 2\sqrt{n-1} \rceil$.
	\item We determined the $k$-domination number for $T_{2n}(W)$, showing that $\gamma_k(T_{2n}(W)) = 2k$ for the specified ranges of $k$.
\end{enumerate}

These results contribute to the structural understanding of non-distance-regular graphs and provide a foundation for further exploration of other resolving parameters within this class of Toeplitz graphs.
%%%%%%%%%%%%%%%%%%%%%%%%%%%%%%%%%%%%%%%%%%%%%%%%%%%%%%%%%%%%%%%%%%%%%%%%%%%%%%%%%
%%%%%%%%%%%%%%%%%===== Funding=======%%%%%%%%%%%%%%%%%%%%%%%%%%%%%%%%%%%%%%%%%%%%
\section*{Funding and Conflict of interest} 
The authors have not disclosed any funding and declare no conflict of interest.
%%%%%%%%%%%%%%%%%%%%%%%%%%%%%%%%%%%%%%%%%%%%%%%%%%%%%%%%%%%%%%%%%%%%%%%%%%%%%%%%%
%%%%%%%%%%%%%%%%%===== Refrences=======%%%%%%%%%%%%%%%%%%%%%%%%%%%%%%%%%%%%%%%%%%

\end{document}